\documentclass[10pt,a4paper,twoside]{article}
\usepackage{amssymb,amsthm}
\usepackage{amsmath}
\usepackage{setspace}
\usepackage[a4paper,top=33mm, bottom=27mm, left=20mm, right=20mm]{geometry}
\usepackage[small, margin=20pt]{caption}
\usepackage{url}
\usepackage{color}

\usepackage{hyperref}

\usepackage{soul}

\allowdisplaybreaks[4]

\title{Large subgraphs without short cycles}

\author{F. Foucaud\footnote{\noindent Department of Mathematics, University of Johannesburg, Auckland Park 2006, South Africa; LAMSADE - CNRS UMR 7243, PSL, Universit\'e Paris-Dauphine, 75775 Paris, France. florent.foucaud@gmail.com}\and M. Krivelevich\footnote{\noindent School of Mathematical Sciences, Raymond and Beverly Sackler Faculty of Exact Sciences, Tel Aviv University, Tel Aviv, 69978, Israel. krivelev@post.tau.ac.il. Research supported in part by USA-Israel BSF Grant 2010115 and by grant 912/12 from the Israel Science Foundation.} \and G. Perarnau\footnote{\noindent School of Computer Science, McGill University, 845 Sherbrooke Street West, Montreal, Quebec, Canada H3A 0G4. guillem.perarnaullobet@mcgill.ca.}}
\date{\today}

\theoremstyle{plain}
\newtheorem{theorem}{Theorem}
\newtheorem{lemma}[theorem]{Lemma}
\newtheorem{proposition}[theorem]{Proposition}
\newtheorem{corollary}[theorem]{Corollary}
\newtheorem{observation}[theorem]{Observation}

\theoremstyle{definition}
\newtheorem{conjecture}[theorem]{Conjecture}

\newtheorem{definition}[theorem]{Definition}
\newtheorem*{acknowledgement}{Acknowledgements}

\theoremstyle{definition}

\newtheorem*{claim}{Claim}

\renewenvironment{proof}[1][Proof]{\begin{trivlist}
\item[\hskip \labelsep {\textit{#1}.}]}{\qed\end{trivlist}}

\newenvironment{proofof}[1][Proofof]{\begin{trivlist}
\item[\hskip \labelsep {\textit{#1}}.]}{\qed\end{trivlist}}

\newcommand{\rst}[1]{\ensuremath{{\mathbin\upharpoonright}%
\raise-.5ex\hbox{$#1$}}} 

\newcommand{\A}{\mathcal{A}}

\newcommand{\F}{\mathcal{F}}
\newcommand{\Feven}{\mathcal{F}_r^{\text{even}}}

\newcommand{\G}{\mathcal{G}}
\newcommand{\E}{\mathbb{E}}
\renewcommand{\P}{\mathcal{P}}

\newcommand{\eps}{\varepsilon}
\newcommand{\Bin}{\mbox{Bin}}

\newcommand{\ex}{\text{ex}}

\begin{document}

\pagenumbering{arabic}

\setcounter{section}{0}

\maketitle

\onehalfspace

\begin{abstract}
We study two extremal problems about subgraphs excluding a family $\F$ of graphs. i) Among all graphs with $m$ edges, what is the smallest size $f(m,\F)$ of a largest $\F$--free subgraph? ii) Among all graphs with
minimum degree $\delta$  and maximum degree $\Delta$, what is the smallest
minimum degree $h(\delta,\Delta,\F)$ of a spanning $\F$--free subgraph
with largest minimum degree? These questions are easy to answer for
families not containing any bipartite graph. We study the case where $\F$ is composed of all even cycles of length at most $2r$, $r\geq 2$. In this case, we give bounds on $f(m,\F)$ and $h(\delta,\Delta,\F)$ that
are essentially asymptotically tight up to a logarithmic
factor. In particular for every graph $G$, we show the existence of subgraphs with arbitrarily high girth, and with either many edges or large minimum degree. These subgraphs are created using probabilistic embeddings of a graph into extremal graphs.
\end{abstract}

\section{Introduction}\label{sec:intro}

Let $G=(V,E)$ be a simple undirected graph with $|V|=n$ vertices. If
$F$ is a given graph, then we say that $G$ is $F$--free if there is no
subgraph of $G$ isomorphic to $F$. The problem of determining the
\emph{Tur\'an number} with respect to $n$ and $F$, i.e. the largest
size of an $F$--free graph on $n$ vertices, has been extensively
studied in the literature. This is the same as determining the size of
a largest $F$--free spanning subgraph of $K_n$, the complete graph on
$n$ vertices. The latter quantity is denoted by $\ex(K_n,F)$. We can extend this notion in a natural way: for every
graph $G$, let $\ex(G,F)$ denote the largest size of an $F$--free
subgraph of $G$. Also, for a family $\F$ of graphs, we say that $G$ is
$\F$--free if $G$ does not contain any graph from $\F$, and denote by
$\ex(G,\F)$ the largest size of an $\F$--free subgraph of $G$.

In this paper, we provide lower bounds for $\ex(G,\F)$ in terms of
different graph parameters of $G$.  If $\F$ does not contain any bipartite graph, it is
easy to provide tight bounds for $\ex(G,\F)$. Therefore, it is more
interesting to study the behavior of $\ex(G,\F)$ when $\F$ contains bipartite graphs. We mostly address the case of even cycles $\F=\{C_4,C_6,\dots,C_{2r}\}$, with $r\geq 2$.

In the first part of the paper, we derive a lower bound for $\ex(G,\{C_4,C_6,\dots,C_{2r}\})$ in terms of the size of $G$.
In the second part of the paper, we study the largest minimum degree of a $\{C_4,C_6,\dots,C_{2r}\}$--free spanning subgraph of $G$ in terms of the maximum and minimum degrees of $G$. In both cases, for every graph $G$, we show the existence of subgraphs with either many edges or large minimum degree, and arbitrarily high girth.

As far as we know, this is the first study of these extremal problems.\\[-2mm]

\noindent\textbf{Key definitions.} Let $F$ be a fixed graph. We define
$f(m,F)$ as the smallest possible size of a largest $F$--free subgraph
of a graph with $m$ edges, that is,
$$
f(m,F):= \min \{ \ex(G,F):\, |E(G)|=m\}\;.
$$
More generally, for a family $\F$ of graphs, we define
$$
f(m,\F):= \min \{ \ex(G,\F):\, |E(G)|=m\}\;.
$$

Our second key definition is $h(\delta,\Delta,F)$, the smallest
possible minimum degree of an $F$--free spanning subgraph $H$ of a
graph $G$ with minimum degree $\delta$ and maximum degree $\Delta$
such that $\delta(H)$ is maximized. If we define $d(G,F)=\max\{\delta(H):\, V(H)=V(G),\, H\subseteq G \text{ and } F\not\subset H\}$, then we can write
$$
h(\delta,\Delta,F):= \min\{d(G,F):\,\Delta(G)=\Delta \text{ and }  \delta(G)=\delta\}\;.
$$

As an illustration, consider the case $F=P_3$, the path on three vertices. Since every graph $H$ with minimum degree at least~2 contains a copy of $P_3$, for any $\delta$ and $\Delta$ with $\Delta\geq\delta\geq 2$, we have $h(\delta, \Delta, P_3)=1$.

Also, for a family $\F$ if we define $d(G,\F)=\max\{\delta(H):\, V(H)=V(G),\, H\subseteq G \text{ and } \forall F\in \F,\, F\not\subset H\}$, then
$$
h(\delta,\Delta,\F):= \min\{d(G,\F):\,\Delta(G)=\Delta \text{ and }  \delta(G)=\delta\}\;.
$$

\noindent\textbf{Thomassen's conjecture.} A related problem for the
girth and the average degree was stated by
Thomassen~\cite{t1983}. Similarly to our definitions of $f(m,\F)$ and
$h(\delta,\Delta,\F)$, we can define the smallest possible average
degree of an $\F$--free subgraph $F$ of a graph $G$ with average
degree $\overline{d}(G)\geq d$ such that $\overline{d}(F)$ is maximized.
Formally, if we define $d'(G,\F)=\max\{\overline{d}(F):\, F\subseteq G \text{ and } \forall F\in \F,\, F\not\subset F\}$, then:
$$
f'(d,\F):=  \min\{d'(G,\F):\,\overline{d}(G)\geq d\}\;.
$$

A reformulation of Thomassen's conjecture is to say that for every fixed $g\geq 4$, $f'(d,\{C_3,\ldots,C_{g-1}\})\to\infty$ when $d\to\infty$. If $g=4$, it is easy to check that $f'(d,\{C_3\})=d/2$.  For $g=6$, K{\"u}hn and Osthus~\cite{ko2004} showed that $f'(d,\{C_3,C_4,C_5\})=\Omega(\sqrt[3]{\log\log d})$, thus confirming the conjecture for this case.\\[-2mm]

Observe that one can change the average degree for the minimum degree in the definition of $d'(G,F)$, since a graph with average degree $\delta$ has a subgraph with minimum degree at least $\delta/2$.
It is interesting to notice that Thomassen's conjecture is true when restricted to graphs of large minimum degree $\delta$ with respect to the maximum degree $\Delta$. In this case, a relatively straightforward application of the local lemma shows that for any family of graphs $\F$ there exists a spanning $\F$-free subgraph with minimum degree $\Omega(\delta\cdot\Delta^{-m(\F)})$, where $m(\F)=\max_{F\in \F} \frac{|V(F)|-2}{|E(F)|-1}$.\footnote{See Lemma~4 of~\cite{ko2004} for a similar statement with $\F=\{C_4\}$.} One of the goals of this paper is to improve this lower bound when $\F$ is composed of cycles.

\noindent\textbf{Families of non-bipartite graphs.} It is rather easy to
determine both functions $f(m,\F)$ and $h(\delta,\Delta,\F)$
asymptotically when $\chi(\F)>2$ (there are no bipartite graphs in $\F$), as shown by the following
proposition (a proof is provided in Section~\ref{sec:prop1}).

\begin{proposition}\label{prop:non_bip}
 For every graph $G$ and for every $k\geq 3$ there exists a
 $(k-1)$--partite subgraph $H$ of $G$ such that for every $v\in
 V(G)$,
 $$
 d_{H}(v)\geq  \left(1-\frac{1}{k-1}\right)d_{G}(v)\;.
 $$
Moreover, for every $\F$ with $\chi(\F)=k$, we have
$$
f(m,\F) = \left(1-\frac{1}{k-1}+o_m(1)\right)m\;,
$$
and
$$
h(\delta,\Delta,\F) = \left(1-\frac{1}{k-1}+o_\delta(1)\right)\delta\;.
$$
\end{proposition}
~\\[-2mm]

\noindent\textbf{Even cycles.} It remains to study $f(m,\F)$ and
$h(\delta,\Delta,\F)$ when $\F$ contains a bipartite graph. In this paper, we focus on
the cases $\F=\F_r=\{C_3,C_4,\dots,C_{2r+1}\}$ and $\F=\Feven =\{C_4,C_6,\dots,C_{2r}\}$ for some $r \geq 2$.
We define
$$
f(m,r)=f(m,\Feven)\;,
$$
and
$$
h(\delta,\Delta,r)=h(\delta,\Delta,\Feven)\;.
$$

Using Proposition \ref{prop:non_bip} one can easily show that every family $\F$ composed of $\Feven$ and other graphs with chromatic number at least $3$ satisfies
\begin{align}\label{eq:equiv1}
f(m,\F)&=\Theta(f(m,r))\\
h(\delta,\Delta,\F)&=\Theta(h(\delta,\Delta,r))\;.\nonumber
\end{align}
This provides us a way to transfer the results obtained for $\Feven$ to $\F_r$.

Denoting by $K_{s,t}$ the complete bipartite graph with parts of size
$s$ and $t$, our first result is:

\begin{theorem}\label{thm:edges}
For every $r\geq 2$ there exists $c=c(r)>0$ such that for
every  large enough $m$,
$$
f(m,r)\geq  \frac{c}{\log{m}} \min_{k\mid m} \text{\emph{$\ex$}}(K_{k,m/k},\text{\emph{$\Feven$}} )\;.
$$
\end{theorem}

Observe that for every $k$ dividing $m$, we have
$$
f(m,r)\leq \ex(K_{k,m/k},\Feven)\;,
$$ since $K_{k,m/k}$ has $m$ edges. Thus, Theorem~\ref{thm:edges} is
asymptotically tight up to a logarithmic factor. This theorem is proved
in Section~\ref{sec3}.

The value of $\ex(K_{k,m/k},\Feven)$ is not known for
general $k$ and $r$. Some results for small values of $r$ can be found in~\cite{cs1991,g1997}.
The results in~\cite{nv2005} imply the following explicit upper bound
$$
f(m,r)=O(m^{(r+1)/2r})\;.
$$

The case $r=2$ of the above problem appears to be more accessible. In particular, combining the upper bound of Kov{\'a}ri, S{\'o}s and Tur{\'a}n~\cite{kst1954} and Reiman~\cite{r1958}, and the lower bound provided by Erd\H os and R\'enyi~\cite{er1962} we have $\ex(K_{n,n},C_4)=\Theta(n^{3/2})$ (see Chapter 6.2 of \cite{bollobas} for a discussion). Here we derive the following corollary (proved in Section~\ref{sec3}):

\begin{corollary}\label{cor:edges_C4}
There exists a constant $c>0$ such that for every  large enough  $m$,
$$
f(m,2)\ge\frac{cm^{2/3}}{\log{m}}\;.
$$
\end{corollary}

We remark that by applying a standard double counting argument
of~\cite{kst1954}, one can show that
$\ex(K_{m^{1/3},m^{2/3}},C_4)=O(m^{2/3})$. Hence we obtain
$f(m,2)=\tilde{\Theta}(m^{2/3})$ (where the $\tilde{\Theta}$
notation neglects polylogarithmic terms). Then, by~\eqref{eq:equiv1}
$$
f(m,\{C_3,C_4,C_5\})=\tilde{\Theta}(m^{2/3})\;.
$$

By using~\eqref{eq:equiv1} with $\F=\F_r$, Theorem~\ref{thm:edges} provides a lower bound on $f(m,\F_r)$, which implies the existence of large subgraphs with high girth.
We also provide a general lower bound for $h(\delta,\Delta,r)$, proved in Section~\ref{sec:deg_C2r} (where $g(G)$ denotes the girth of graph $G$):

\begin{theorem}\label{thm:deg}
Let $G$ be a graph with minimum degree $\delta$ and (large enough) maximum degree~$\Delta$ such that
$\emph{$\ex$}(K_\Delta,\F_r)\delta\geq \alpha \Delta^2\log^4{\Delta}$ for some large constant $\alpha>0$. Then, for every $r\geq 2$ there exists a spanning subgraph $H$ of $G$ with $g(H)\geq 2r+2$ and $\delta(H)\geq\frac{c\cdot \emph{\ex}(K_\Delta,\F_r) \delta}{\Delta^2 \log{\Delta}}$, for some small constant $c>0$.
In particular, under the above conditions on $\delta$ and $\Delta$,
$$
h(\delta,\Delta,r)\geq h(\delta,\Delta,\F_r) \geq \frac{c\cdot \text{\emph{\ex}}(K_\Delta,\F_r) \delta}{\Delta^2 \log{\Delta}}  \;.
$$
\end{theorem}

The existence of a graph with a given number of vertices, many edges and large girth is one of the
most interesting open problems in extremal graph theory (see Section~$4$ in~\cite{fs2013}).
The best known result that holds for any value of $r\geq 2$ was given by Lazebnik, Ustimenko and Woldar in~\cite{luw1995}. They showed that $\ex(K_n,\F_r)=\Omega(n^{1+\frac{2}{3r-2}}$). Using that, we can obtain the following explicit corollary of Theorem~\ref{thm:deg}.
\begin{corollary}\label{cor:deg_C2r}
For every $\Delta$ and $\delta$ such that $\ex(K_\Delta,\F_r)\delta\geq \alpha \Delta^2\log^4{\Delta}$ for some large constant $\alpha>0$, we have
$$
h(\delta,\Delta,r)\geq h(\delta,\Delta,\F_r)= \Omega\left(\frac{\delta}{\Delta^{1-\frac{2}{3r-2}} \log{\Delta}}  \right)\;.
$$
\end{corollary}

The upper bound for the Tur\'an number of even cycles
$\ex(K_d,C_{2r})=O(d^{1+1/r})$,  shown by Bondy and
Simonovits~\cite{bs1974}, implies that
\begin{align}\label{eq:UBexC2r}
h(d,d,r) \leq \frac{2}{d}\ex(K_{d+1},C_{2r}) = O(d^{\frac{1}{r}})\;.
\end{align}
Since the above cited upper bound for $\ex(K_{d+1},C_{2r})$ is
conjectured to be of the right order, Corollary~\ref{cor:deg_C2r} is
probably not tight.

For the particular case $r=2$ we can derive a better bound.  Erd\H os, R\'enyi and S\'os~\cite{ers1966} and Brown~\cite{b1966} showed that for every prime $p$ there exists a $C_4$--free $p$-regular graph with $p^2-1$ vertices. By the density of primes, for every $n$ there exists a graph of order at most $2n$ satisfying the former properties. Thus, we can obtain the following corollary of Theorem~\ref{thm:deg}:

\begin{corollary}\label{cor:deg_C4}
For every $\Delta$ and $\delta$ such that $\delta\geq \alpha \sqrt{\Delta}\log^4{\Delta}$ for some large constant $\alpha>0$, we have
$$ h(\delta,\Delta,2)\geq h(\delta,\Delta,\{C_3,C_4,C_5\}) = \Omega\left(\frac{\delta}{\sqrt{\Delta} \log{\Delta}}\right)\;.
$$
\end{corollary}

This corollary is tight up to a logarithmic factor, as will be shown in
Proposition~\ref{prop:deg_C4_tight} (Section~\ref{sec:deg_C2r}). The
condition on $\delta$ and $\Delta$ is also tight up to a logarithmic factor since, if $\Delta\geq\delta^2$, any spanning $C_4$--free subgraph $H$ of $K_{\delta,\Delta}$ satisfies $\delta(H)\le 1$.
Similar results can be derived for $r=3,5$, since there exist graphs with girth at least $8$ and $12$ respectively and large minimum degree~\cite{luw1999,lv2005}.\\[-2mm]

\noindent\textbf{Outline.} We begin with some preliminary
considerations in Section~\ref{sec:tools}. We then prove
Theorem~\ref{thm:edges} and Corollary~\ref{cor:edges_C4} in
Section~\ref{sec3}. Section~\ref{sec:deg_C2r} is devoted to the proof
of Theorem~\ref{thm:deg} and Proposition~\ref{prop:deg_C4_tight}. We
conclude with some remarks and open questions in Section~\ref{sec:conc}.

\section{Preliminaries}\label{sec:tools}

For every vertex $v\in V$ let $N_G(v)$ denote the set of neighbors of $v$ in $G$ and $d_G(v)=|N_G(v)|$. If the graph $G$ is clear from the context we will denote the above quantities by $N(v)$ and $d(v)$, respectively. We use $\log(x)$ to denote the natural logarithm of $x$.

\subsection{Proof of Proposition~\ref{prop:non_bip}}\label{sec:prop1}

We start by proving Proposition~\ref{prop:non_bip}. We will use the
well-known Erd\H os-Stone-Simonovits theorem~\cite{es1946,es1966}: for every graph $F$
with $\chi(F)=k$,
\begin{align}\label{eq:ES}
\ex(K_n,F) = \left(1-\frac{1}{k-1} +o(1)\right)\binom{n}{2}\;.
\end{align}
\begin{proofof}[Proof of Proposition~\ref{prop:non_bip}]
 Let $G$ be a graph with $m$ edges  and consider a $(k-1)$--partition $\P=\{P_1,P_2, \dots ,P_{k-1}\}$  of $V(G)$ that maximizes the number of edges in the subgraph $G_{\P}=G\setminus (G[P_1]\cup \dots
\cup G[P_{k-1}])$. Then, we claim that for every $v\in V(G)$, $d_{G_{\P}}(v)\geq \left(1-\frac{1}{k-1}\right)d_G(v)$.

For the sake of contradiction, suppose that there is a vertex $v\in P_i$, $i\in [k-1]$, with degree in $G_{\P}$ less than
$\left(1-\frac{1}{k-1}\right)d_G(v)$. Then, there are more than $\frac{d_G(v)}{k-1}$ neighbors of $v$
in $P_i$ and there is a part $P_j$, $j\neq i$, with at
most $\frac{d_G(v)}{k-1}$ neighbors of $v$. Moving $v$ from $P_i$ to $P_j$
increases the number of edges in $G_{\P}$ by at least one, which gives a contradiction by the choice of $\P$.

Clearly, $G_{\P}$ does not contain any copy of $F$ with $\chi(F)\geq k$, and thus it does not contain any graph in $\F$.
Observe also that $G_{\P}$ has at least $\left(1-\frac{1}{k-1}\right)m$ edges and minimum degree $\delta(G_{\P}) \geq \left(1-\frac{1}{k-1}\right)\delta(G)$.

 Notice that $f(m,\F)\leq f(m,F)$ and $h(\delta,\Delta,\F)\leq h(\delta,\Delta,F)$ for every $F\in \F$. Let $F\in \F$ be such that $\chi(F)=k$. The upper bound for $f(m,F)$ follows from (\ref{eq:ES}) by choosing $n$ for which $\binom{n-1}{2}<m\le\binom{n}{2}$, and then by taking $G$ to be any subgraph of $K_n$ with exactly $m$ edges. For the upper bound on $h(\delta,\Delta,F)$ (assuming $\delta\ge 2$) take $\Delta$ disjoint copies of $K_{\delta+1}$, add a new vertex $v$ and connect it to one vertex from each of the cliques $K_{\delta+1}$. If a subgraph $H$ of the so obtained graph $G$ is $F$-free, then the subgraph of $H$ spanned by the vertex set of each of the cliques $K_{\delta+1}$ is $F$-free as well, thus implying by (\ref{eq:ES}) that $\delta(H)\le  \left(1-\frac{1}{k-1} +o(1)\right)\delta$.
\end{proofof}

\subsection{Useful definitions}

The following definitions will be useful in our proofs.

\begin{definition}\label{def:Gprime}
For every graph $G$, every $\G$ with $V(\G)=[\ell]$ and every vertex labeling
$\chi: V(G)\to [\ell]$ we define the spanning subgraph
$H'_{(\chi,\G)}\subseteq G$ as the subgraph with vertex set $V(G)$
where an edge $e=uv$ is present if and only if $uv\in E(G) \text{ and
} \chi(u)\chi(v)\in E(\G)$.
\end{definition}

\begin{definition}\label{def:Gstar}
For every graph $G$, every $\G$ with $V(\G)=[\ell]$ and every vertex labeling
$\chi: V(G)\to [\ell]$ we define the spanning subgraph
$H^*_{(\chi,\G)}\subseteq G$ as the subgraph with vertex set $V(G)$ such that an edge $e=uv$ is present in $H^*$ if all the following
properties are satisfied: \renewcommand{\labelenumi}{\roman{enumi})}
\begin{enumerate}
\item $uv\in E(G) \text{ and } \chi(u)\chi(v)\in E(\G)$, that is $e\in E(H')$,
\item for every $w\neq v$, $w\in N_G(u)$, we have $\chi(w)\neq \chi(v)$, and
\item for every $w\neq u$, $w\in N_G(v)$, we have $\chi(w)\neq \chi(u)$.
\end{enumerate}
\renewcommand{\labelenumi}{\arabic{enumi})}
\end{definition}

These two definitions are crucial for our proofs. We will use them to randomly embed a (large) graph $G$ into a fixed extremal graph $\G$ satisfying certain conditions, therefore creating a subgraph of $G$. The way in which we construct the subgraph will allow us to prove that it preserves some of the properties of $\G$.

The concept of frugal coloring was introduced by Hind, Molloy and Reed in~\cite{hmr1997}. We say that a proper coloring $\chi: V(G)\to [\ell]$ is \emph{$t$-frugal} if  for every vertex $v$ and every color $c\in [\ell]$,
$$
|N_G(v)\cap \chi^{-1}(c)|\leq t\;,
$$
that is, there are at most $t$ vertices of the same color in the neighborhood of each vertex.
For instance, a $1$--frugal coloring of $G$ is equivalent to a proper coloring of $G^2$.

\subsection{Probabilistic tools}

Here we state some (standard) lemmas we will use in the proofs.

\begin{lemma}[Chernoff inequality for binomial
distributions~\cite{as2008}]\label{lem:chernoff}
	Let $X\sim \Bin(N,p)$ be a Binomial random variable. Then for all $0<\eps<1$,
\begin{enumerate}
 \item  $\Pr(X\leq (1-\eps)Np)< \exp\left(-\frac{\eps^2}{2}Np\right).$
\item   $\Pr(X\geq (1+\eps)Np)< \exp\left(-\frac{\eps^2}{3}Np\right).$
\end{enumerate}
\end{lemma}

Let $L:S^T\to \mathbb{R}$ be a functional. We say that $L$ satisfies the Lipschitz condition if for every $g$ and $g'$ differing in just one coordinate from the product space $S^T$, we have
$$
|L(g)-L(g')|\leq 1\;.
$$

\begin{lemma}[Azuma inequality, Theorem $7.4.2$ in \cite{as2008}]\label{lem:azuma}
Let $L$ satisfy the Lipschitz condition relative to a gradation of length $l$ (i.e. $|T|=l$). Then for all $\lambda>0$
\begin{enumerate}
 \item  $\Pr(L\leq \mathbb{E}(L)-\lambda \sqrt{l})< e^{-\frac{\lambda^2}{2}},$
 \item  $\Pr(L\geq \mathbb{E}(L)+\lambda \sqrt{l})< e^{-\frac{\lambda^2}{2}}.$
\end{enumerate}
\end{lemma}

\begin{lemma}[Weighted Lov\'asz Local Lemma~\cite{mr2002}]\label{lem:WLLL}
Let $\A=\{A_1 ,\dots, A_N\}$ be a set of events and let $H$ be a dependency graph for $\A$.

If there exist   weights $w_1,\ldots,w_N\ge 1$ and a real
$p\leq\tfrac{1}{4}$ such that for each $ i\in [N]$:
\begin{enumerate}
\item $\Pr(A_i)\le p^{w_i}$, and
\item $\sum_{j:\,ij\in E(H)}(2p)^{w_j}\le\frac{w_i}{2}$\;,
\end{enumerate}
then
\begin{align*}
\Pr \left(\bigcap_{i=1}^N \overline{A_i}\right)>0\,.
\end{align*}
\end{lemma}

\section{Subgraphs with large girth and many edges}\label{sec3}

This section is devoted to the proofs of Theorem~\ref{thm:edges} and
Corollary~\ref{cor:edges_C4}.

\begin{proofof}[Proof of Theorem~\ref{thm:edges}]
Let $G$ be a graph with $m$ edges. Define $V_1$ to the set of vertices of
$G$ with degree at least $2\sqrt{m}$, and $V_2=V\setminus V_1$.
Observe that
$$
|V_1|\le 2m/2\sqrt{m}=\sqrt{m}
$$
 and thus $V_1$ spans at most $m/2$ edges. Recall that $\Feven=\{C_4,C_6,\dots, C_{2r}\}$ is the family of all even cycles of length at most $2r$.
 In order to find an $\Feven$--free subgraph with many edges, we will remove all the edges inside $V_1$ and only look
at the edges between $V_1$ and $V_2$, and the edges inside $V_2$. We will split the proof into two cases in terms of the number of edges between $V_1$ and $V_2$.\\[-2mm]

\textbf{Case 1: \boldmath{$e(V_1,V_2)\ge m/4$}.}
Consider the following partition of $V_1$ into sets $U_1,\dots, U_s$,
where $U_p$ contains all vertices of $V_1$ whose degree into $V_2$ is
between $2^p$ and $2^{p+1}$.

Notice that $s\leq \log_2{m}$.  Hence, there is a subset $U_q$ incident
to at least $m/4\log_2{m}$ edges leading to $V_2$. Let $|U_q|=k$, then
$k=\Omega(m/(2^q\log m))$ and $k=O(m/2^q)$ since all the degrees in
$U_q$ are between $2^q$ and $2^{q+1}$.  Let $u_1,\dots,u_k$ be the
vertices of $U_q$.

Fix an $\Feven$--free bipartite graph $\G$ with parts  $A=\{a_1,\dots,a_k\}$ and
$B=\{b_1,\dots,b_{m/k}\}$, and $\ex(K_{k,m/k},\Feven)$ edges, and let $\chi: U_q\cup V_2
\rightarrow A\cup B$ be the $(A\cup B)$-coloring defined by $\chi(u_i)=a_i$ for any $u_i\in U_q$ and by $\chi(v)=b$ for any $v\in V_2$, where $b\in B$ is chosen independently and uniformly at random.

We say that $b\in B$ is {\it good} for $u_i\in U_q$ if:
\begin{enumerate}
 \item  $(a_i,b)\in E(\G)$, and
 \item  exactly one neighbor of $u_i$ in $G$ is colored $b$ by $\chi$.
\end{enumerate}

Let $G_1$ be the subgraph of $G$ with vertex set $U_q\cup V_2$ that only contains the edges
$e=(u_i,v)\in E(G)$, where $u_i\in U_q$, $v\in V_2$, and $\chi(v)=b$ is
good for $u_i$.
We claim that $G_1$ is an $\Feven$--free graph. Assume that
there is a cycle of length $l$ of even length in $G_1$ with $l\leq 2r$ and let
$w_1,\dots,w_{l}$ be its vertices. Observe that there should be at
least one repeated color in the vertices of the cycle, otherwise $\G$
would contain a $C_{l}$. Let $ x,y\in [l]$ ($x<y$) such that
$\chi(w_x)=\chi(w_y)$ and for every $x<z<y$, $\chi(w_x)\neq
\chi(w_z)$. Besides, since $\chi$ is $1$-frugal in $G_1$ (in the sense that for each  vertex of $G_1$, all its neighbors are colored differently by $\chi$), $y-x\geq
3$. Thus, there exists a cycle of length at least $3$ and at most $2r$ in $\G$, a
contradiction since $\G$ is bipartite and, since it is $\Feven$--free, it has no even cycles of length at most $2r$.

Let us compute the expected size of $G_1$. Fix $u_i\in U_q$. Then, the probability that a given edge $e=(u_i,v)\in E(G)$ exists in $G_1$ is
$$
\Pr(e\in E(G_1)) \geq \frac{d_{\G}(a_i)}{m/k} \left(1-\frac{1}{m/k}\right)^{2^{q+1}}= \Omega\left( \frac{d_{\G}(a_i)k}{m} e^{-\frac{2^q k }{m}}\right)  =  \Omega\left( \frac{d_{\G}(a_i)k}{m}\right)\;,
$$
since $2^q k = O(m)$.

Thus, the expected degree of $u_i$ in $G_1$ is of order
$$
\E(d_{G_1} (u_i)) =\Omega\left(2^q\frac{d_{\G}(a_i)k}{m}\right)\;.
$$

Recall that $\sum_{a_i\in A} d_{\G}(a_i)=\ex(K_{k,m/k},\Feven)$, hence the
expected number of edges of $G_1$ is of order\linebreak $2^q \ex(K_{k,m/k},\Feven)k/m$. Since
$2^qk\ge cm/\log m$ for some small constant $c>0$, the expected number of edges is of order at least
$$
\min_{k\mid m} 2^q \frac{\ex(K_{k,m/k},\Feven)k}{m} \geq \frac{c}{\log m} \min_{k\mid m} \ex(K_{k,m/k},\Feven) \;.
$$
~\\[-2mm]

\textbf{Case 2: \boldmath{$e(V_1,V_2)<m/4$}.}
Then $e(V_2)>m/4$, and all the degrees in $G[V_2]$ are less than $2\sqrt{m}$. For the sake of convenience we will assume that $2\sqrt{m}$ is an integer.

We will find a large $\Feven$--free graph inside $V_2$. For this, let $\G$ be an
$\Feven$--free graph on $2\sqrt{m}$ vertices with the largest possible number of edges.
Assume $V(\G)=[2\sqrt{m}]$ and let $\chi: V_2\rightarrow V(\G)$ be a random labeling of the vertices of $V_2$.
Observe that using Proposition~\ref{prop:non_bip} we can select a bipartite subgraph $\G'$ of $\G$ with at least $\ex(K_{2\sqrt{m}},\Feven)/2$ edges (in particular, $\G'$ is $\F_r$--free).

Consider the graph $H^*=H^*_{(\chi,\G')}$ from Definition~\ref{def:Gstar} applied to the induced subgraph $G[V_2]$. Since $\G'$ is $\F_r$--free, $H^*$ is also $\F_r$--free.
For each edge $e=uv$ of $G$ spanned by $V_2$, its probability to belong to $H^*$ is  at least
$$
\frac{1}{2\sqrt{m}}\sum_{i\in V(\G')}\frac{d_{\G'}(i)}{2\sqrt{m}}\left(1-\frac{1}{2\sqrt{m}-2}\right)^{d_G(u)+d_G(v)-2}
= \Omega\left(\frac{\ex(K_{2\sqrt{m}},\Feven)}{m}\right)\;,
$$
since $d_G(u),d_G(v)\leq 2\sqrt{m}$. To justify the above estimate, first choose a label $i$ for $u$, then require a label $j$ of $v$ to be one of the neighbors of $i$ in $\G'$, and finally for each of the neighbors of $u$ and $v$ in $G$ choose a label different from $i$ and $j$.

Thus, we expect
$$
\Omega\left(\ex(K_{2\sqrt{m}},\Feven)\right)= \Omega\left(\ex(K_{\sqrt{m},\sqrt{m}},\Feven)\right)=\omega\left(\frac{\ex(K_{\sqrt{m},\sqrt{m}},\Feven)}{\log{m}}\right)
$$
edges in the $\Feven$--free subgraph $H^*$ of $G$.
\end{proofof}

\medskip

We now prove Corollary~\ref{cor:edges_C4}.

\begin{proofof}[Proof of Corollary~\ref{cor:edges_C4}]
By Theorem~\ref{thm:edges} we have
$$ f(m,2)\geq \frac{c}{\log{m}} \min_{k\mid m}
\ex(K_{k,m/k},C_{4})\;,
$$ for some small constant $c>0$. By the symmetry of $K_{s,t}$ and
$K_{t,s}$, we may assume without loss of generality that the minimum
is attained when $k\leq \sqrt{m}$.

We provide two constructions. First, the star of size $m/k$, which is a $C_4$--free graph,
shows that $\ex(K_{k,m/k},C_{4}) \geq m/k$.

On the other hand, one can construct a $C_4$--free graph by selecting
a subgraph of a larger $C_4$--free graph. Let $G_1$ be a largest
$C_4$--free subgraph of $K_{m/k,m/k}$. We construct a $C_4$--free
bipartite graph $G_2$ by keeping the $k$ vertices with highest degrees
in one of the parts of $G_1$. Then $G_2$ is a subgraph of $K_{k,m/k}$
and has at least $\frac{k^2}{m}\ex(K_{m/k,m/k},C_4)= \Omega (\sqrt{mk})$
edges.

Hence, for every $k$,
$$
\ex(K_{k,m/k},C_4)\geq \max\{m/k,\Omega(\sqrt{mk})\}= \Omega\left(m^{2/3}\right)\;.
$$
\end{proofof}

\section{Subgraphs with large girth and large minimum degree}\label{sec:deg_C2r}

We devote this section to the proof of  Theorem~\ref{thm:deg}. Before proving the theorem, let us state an important observation and an auxiliary lemma.

In this proof we will use a graph $\G$ that satisfies the following: it has $|V(\G)|=\ell\in (k,2k)$ (where  $k=2e^4\Delta$) vertices, girth at least $2r+2$ and minimum degree at least $q=\ex(K_k,\F_r)/2k$.

The existence of such $\G$ is provided by the following observation.
\begin{observation}\label{obs:graph}
Let $\G'$ be an $\F_r$-free graph on $2k$ vertices with $\emph{\ex}(K_{2k},\F_r)$ edges. Then we claim that $\G'$ has a subgraph $\G$ satisfying the desired properties. Assume it is not the case, namely all its subgraphs of size at least $k$ have minimum degree smaller than $q$. It is also clear that $\emph{\ex}(K_{2k},\F_r)\geq 2\emph{\ex}(K_{k},\F_r)$. By iteratively removing vertices of degree smaller than $q$, we can obtain an $\F_r$-free graph on $k$ vertices with at least $\emph{\ex}(K_{2k},\F_r)-qk= \emph{\ex}(K_{2k},\F_r)-\emph{\ex}(K_k,\F_r)/2 \geq 3\cdot\emph{\ex}(K_k,\F_r)/2$ edges, a contradiction.
\end{observation}

In a vertex-colored graph, a cycle is called \emph{rainbow} if all its
vertices have distinct colors. A path is called \emph{maximal
  inner-rainbow} if its endpoints have the same color $i$, but all
other vertices of the path are colored with distinct colors (other
than~$i$). We will use the following lemma:

\begin{lemma}\label{lem:better_frugal}
 Let $G$ be a graph with maximum degree $\Delta$ and minimum
 degree $\delta$, that admits a $t$--frugal coloring $\chi$ without
 rainbow cycles of length at most $2r+1$ and maximal inner-rainbow
 paths of length $l$ for every $3\leq l\leq 2r$. If
 $\delta>129 t^3\log{\Delta}$ and $\Delta$ is large enough, then
 there exists a subgraph $H \subseteq G$ such that
\begin{enumerate}
 \item $\forall v\in V(G)$, $d_{H}(v)\geq \frac{d_G(v)}{4t}$,
and
 \item $g(H)\geq 2r+2$.
\end{enumerate}
\end{lemma}
\begin{proof}
Let the color classes of $\chi$ be $S_1$, $S_2$, $\dots$.
Assign to each edge $e\in E(G)$ a random variable $f(e)$ uniformly
distributed in $(0,1)$.  We construct the following subgraph $H$:
for every pair of color classes $(S_i,S_j)$ of $\chi$, an edge $e$
between $S_i$ and $S_j$ is retained in $H$ if $f(e)$ is less than
$f(e')$ for every $e'$ between $S_i$ and $S_j$ incident with $e$.
Observe that, by construction, $\chi$ is a $1$-frugal coloring of
$H$, that is, the vertices of any pair of color classes of $\chi$
induce a matching.

We claim that deterministically (i.e., with probability 1) $g(H)\geq 2r+2$.  For the sake of contradiction,
suppose that $g(H)<2r+2$ and let $C=(u_1,\dots,u_{l})$ be a
cycle in $H$ with $l<2r+2$. Since $\chi$ does not induce any
rainbow cycle of length at most $l$ in $G$, there exist at least
two vertices of $C$ with the same color. Let $a,b\in[l]$ be such
that, $a<b$, $\chi(u_a)=\chi(u_b)$ and for every $a<j<b$,
$\chi(u_a)\neq \chi(u_j)$. Then $3\leq b-a\leq 2r$ since $\chi$ is a
$1$--frugal coloring in $H$ and $C$ is not rainbow. Then
$u_a,\ldots,u_b$ is a maximal inner-rainbow path with a forbidden
length, a contradiction.

Now, it remains to show that with positive probability the obtained
subgraph $H$ has the desired minimum degree.

Observe first that a given edge $e=(u,v)$, with $u\in S_i$ and $v\in
S_j$, is preserved in $H$ with probability
$$
\frac{1}{d_G(u,S_j)+d_G(v,S_i)-1}\ge \frac{1}{2t-1}\;,
$$
where $d_G(w,S)=|N_G(w)\cap S|$.

For every $v\in V$ consider the random variable $L(v)$ equal to the
degree of $v$ in $H$. We have: $\E(L(v))\geq d_G(v)/2t$. By applying
Azuma's inequality (Lemma~\ref{lem:azuma}) we will now show that with
probability exponentially close to 1, $L(v)$ is large enough.

First of all, observe that $L(v)$ only depends on the edges that
connect a neighbor of $v$ to a vertex of color $\chi(v)$. Let $T_v$
be the set of these edges:
$$
T_v=\bigcup_{u\in N_G(v)} \{vu\} \cup \{uw:\; w\in N_G(u) \text{ and } \chi(v)=\chi(w)\}\;.
$$ Since $\chi$ is $t$-frugal, we have $|T_v|\leq t d_G(v)$. Let
$S=(0,1)$. Then $L(v)$ depends only on a vector in $S^{T_v}$.

If two functions $f,f':S^{T_v}\to\mathbb{R}$ differ only on one edge
of $T_v$, $|L(v)(f(T_v))-L(v)(f'(T_v))|\leq 1$. Thus $L=L(v)$
satisfies the $1$--Lipschitz condition. Let $A_v$ be the event that
$L(v) \leq d_G(v)/4t$.  Since the martingale length is at most
$t d_G(v)$, by setting $\lambda= \frac{\sqrt{d_G(v)}}{4t^{3/2}}$ in
Lemma~\ref{lem:azuma},
$$
\Pr(A_v)=\Pr\left(L(v)\leq \frac{d_G(v)}{4 t}\right) =\Pr\left(L(v) \leq \mathbb{E}(L(v)) - \frac{d_G(v)}{4 t} \right)<
e^{-\frac{d_G(v)}{32 t^3}} < e^{-\frac{\delta}{32t^3}} \;.
$$

Since $A_v$ is influenced only by the edges in $T_v$, it is thus
independent of each $A_u$ unless $u$ is at distance at
most~4 from $v$, and there are at most $\Delta+\Delta(\Delta-1)+\Delta(\Delta-1)^2+\Delta(\Delta-1)^3\leq\Delta^4$ such events.

Notice that,
$$
2\Pr(A_v)\Delta^4< 2e^{-\frac{\delta}{32t^3}}\Delta^4 <\frac{1}{2}\;,
$$
since $\delta> 129 t^3 \log{\Delta}$. Thus, by the Lov\'asz Local Lemma (Lemma~\ref{lem:WLLL})
with $p=\max_{v\in V}\Pr(A_v)$ and $w_{i}=1$ we have
$$
\Pr(\cap_{v\in V}\overline{A_v})>0\;,
$$
and thus there is a way to assign values to $f(e)$ so that $\delta(H)\geq d_G(v)/4t$.
\end{proof}

Now we are ready to prove Theorem~\ref{thm:deg}.
\begin{proofof}[Proof of  Theorem~\ref{thm:deg}]

The idea in this proof is to randomly color the vertices from $G$ with
$\ell$ colors, where $\ell=|V(\G)|$ for the graph $\G$ introduced above. We then consider the subgraph
$H'=H'_{(\chi,\G)}$ from Definition~\ref{def:Gprime} induced by the
coloring and the graph $\G$. We will show that with positive
probability, such a coloring is $t$-frugal in $H'$ and that $H'$ contains neither rainbow cycles of length at most $2r$ nor maximal inner-rainbow paths of length $3\leq l\leq 2r$. The value of $t$ will be set later in the proof.  Then, we will use Lemma~\ref{lem:better_frugal} to obtain the desired subgraph.

Let $\chi$ be a uniformly random coloring of $V(G)$ with $\ell$ colors.
Consider the $\F_r$-free graph $\G$ of order $\ell\geq 2e^4\Delta$ that satisfies $g(\G)\geq
2r+2$ and $\delta(\G)\geq q= \frac{\ex(K_{2e^4 \Delta},\F_r)}{4e^4 \Delta}$ provided by Observation~\ref{obs:graph}. Then construct the spanning subgraph $H'=H'_{(\chi,\G)}$ of $G$.

Since $g(\G)\geq 2r+2$, $\chi$ does not induce any rainbow cycle of length
at most $2r+1$ in $H'$, nor any maximal inner-rainbow path of length at
least~3 and at most~$2r$. Moreover, $\chi$ is a proper coloring of $H'$, since $\G$ has no loops.

We will use the Lov\'asz Local Lemma to show that there is a positive probability that the random
$\ell$-coloring of $G$ satisfies the following properties:
\begin{enumerate}
 \item for every $v\in V(G)$, $d_{H'} (v)=\Omega\left(\frac{qd_G(v)}{\ell}\right)$ and
 \item $\chi$ is $t$-frugal in $H'$.
\end{enumerate}

For this purpose we will define the following events:
\begin{enumerate}
 \item \textbf{Type A:} for each $v\in V(G)$, $A_v$ is the event $d_{H'}(v)\leq \frac{q d_G(v)}{2\ell}$,
 \item \textbf{Type B:} for each vertex $v\in V(G)$ and each set $X=\{x_1,\dots, x_{t+1}\}\subseteq
N_G(v)$, $B_{v,X}$ is the event  $\chi(x_i)=\chi(x_j)$, for every distinct $i,j\in [t+1]$.
\end{enumerate}
Observe that the probability that an edge is retained in $H'$ is at least $q/\ell$. Thus, $\E(d_{H'}(v))\geq q d_G(v)/\ell$. Using Part 1 of Lemma~\ref{lem:chernoff} with $\eps=1/2$, one can check that
$$
\Pr(A_v)= \Pr\left(d_{H'}(v)\leq  \frac{q d_G(v)}{2\ell}\right) \leq \Pr\left(d_{H'}(v)\leq \frac{1}{2}\E(d_{H'}(v))\right)  \leq e^{-\Omega(\E(d_{H'}(v)))}  \leq e^{-\tfrac{c'q d_G(v)}{\ell}}\;,
$$
for some small enough constant $c'>0$.
We also have
$$
\Pr(B_{v,X})=\ell^{-t}\;.
$$

Here, it is convenient to define the auxiliary event $D_e$ as the event $e\in E(H')$. Observe that any event $A_v$ or $B_{v,X}$ can be expressed in terms of the events $D_e$.
\begin{claim}
 Let $v\in V(G)$ and let $F\subset E(G)$ be a set of edges not incident to $v$.
Then, for each $i\in [\ell]$,
$$
\Pr(\chi(v)=i\mid \cup_{f\in F} D_f)=\frac{1}{\ell}\;.
$$
\end{claim}
To prove the claim, observe that all the unveiled information is about
non-incident edges and thus no information about the color of $v$ has
been provided. While information on the existence of the edges in $F$
may affect the degree of $v$, it cannot affect its color.

By the previous claim, the color given to a vertex $v$ depends only on the events that intersect this vertex. Thus, it depends on at most $\Delta$ events of type $A$, precisely the events $A_w$ where $vw\in E(G)$, and on at most $\Delta \binom{\Delta}{t}$ events of type $B$, precisely the events $B_{w,X}$ where $vw\in E(G)$ and $v\in X$.
Moreover, the existence of an edge $e=uv$ in $H'$ depends only on the colors of $u$ and $v$.

Since an event of type $A$ depends on the existence of at most $\Delta$ edges and an event of type $B$ depends on the colors given to a set of $t+1$ vertices, we have at most the number of dependencies given in Table 1.

\begin{table}[ht!]
\centering
    \begin{tabular}{r||c|c}
 		& Type $A$ & Type $B$ \\  & & \\[-0.3cm]
\hline
\hline & & \\[-0.3cm]
     Type $A$  & $2\Delta^2$ & $2\Delta^2\binom{\Delta}{t}$ \\ & & \\[-0.3cm]
    \hline & & \\[-0.3cm]
        Type $B$ & $(t+1)\Delta$ & $(t+1)\Delta\binom{\Delta}{t}$\\ & & \\[-0.1cm]
    \end{tabular}
\caption{Table of dependencies}
\label{tab:results}
\end{table}

Then, by applying the weighted version of the Local Lemma (Lemma \ref{lem:WLLL}) with $p=\ell^{-1}$,
$w_i=\frac{c'q\delta}{\ell\log{\ell}}$ if the weight corresponds to an event of type $A$ and $w_i=t$ if it corresponds to an event of type $B$, we have that a subgraph $H'$ avoiding
all events exists if:
\begin{align*}
 2\Delta^2 e^{-\frac{c'q\delta}{\ell}\left(1-\frac{\log 2}{\log \ell}\right)} +2\Delta^2\binom{\Delta}{t}(2p)^{t} & \leq \frac{c'q\delta}{2\ell\log{\ell}}\mbox{\;, and}\\
 (t+1)\Delta e^{-\frac{c'q\delta}{\ell}\left(1-\frac{\log 2}{\log \ell}\right)}   + (t+1)\Delta\binom{\Delta}{t} (2p)^{t} & \leq \frac{t}{2}\;.
\end{align*}
Recall that by the hypothesis of the theorem, $\ex(K_\Delta,\F_r)\delta\geq \alpha \Delta^2\log^4\Delta$, which in particular implies $c'q\delta\geq \tfrac{c'\alpha}{4 e^4} \Delta \log^4{\Delta} \geq 4\ell \log{\Delta}$, if $\Delta$ is large enough. Set $t=\log{\Delta}$. Then $p^{-1}=\ell\geq 2e^4\Delta= 2\Delta^{1+\frac{4}{t}}$ and the previous inequalities are satisfied if $\Delta$ is large enough.

Therefore, there is a subgraph $H'$  such that for every $v\in V$,
$d_{H'}(v) \geq q\delta/2\ell$ and $H'$ admits a $\log{\Delta}$--frugal coloring with $\ell$ colors, no rainbow cycle of length at most $2r+1$, and no maximal inner-rainbow path of
length at least~3 and at most~$2r$.

Let us check that the minimum degree of $H'$ is large enough to apply Lemma~\ref{lem:better_frugal}. By using that $q= \frac{\ex(K_k,\F_r)}{2k}$, where $k=2e^4\Delta(G)$, that $\ell\in (k,2k)$ and that $\ex(K_{\Delta(G)},\F_r)\delta(G)\geq \alpha (\Delta(G))^2\log^4{\Delta(G)}$, we have
$$
\delta(H')\geq \frac{q\delta(G)}{2\ell}
\geq \frac{\ex(K_{2e^4\Delta(G)},\F_r)\delta(G)}{8e^4\Delta(G) \ell}
\geq \frac{\ex(K_{\Delta(G)},\F_r)\delta(G)}{32e^8(\Delta(G))^2}
\geq \frac{\alpha \log^4{\Delta(G)}}{32e^8} > 129 t^3\log{\Delta(G)}\geq 129 t^3\log{\Delta(H')}\;,
$$
provided that $\alpha$ is large enough.

Then Lemma~\ref{lem:better_frugal} provides a spanning subgraph $H$ of $G$
with $g(H)\geq 2r+2$ and for every $v\in V$,
$$
d_{H}(v) \geq \frac{q\delta/2\ell}{4t} \geq \frac{c\cdot\ex(K_\Delta,\F_r)\delta}{\Delta^2\log{\Delta}}\;,
$$
for some small constant $c>0$.
\end{proofof}

The following proposition shows that Corollary \ref{cor:deg_C4} is tight up to a logarithmic factor.

\begin{proposition}\label{prop:deg_C4_tight}
For every $\delta,\Delta$ satisfying $\Delta\leq\delta^2\leq \Delta^2$, there exists a graph $G$ with minimum degree $\delta$ and maximum degree $\Delta$ such that for every spanning $C_4$--free subgraph $H$ of $G$,
$$
\delta(H)= O\left(\frac{\delta}{\sqrt{\Delta}}\right)\;.
$$
\end{proposition}
\begin{proof}
Let $G$ be the complete bipartite graph with parts $A,B$ of sizes
$\Delta$ and $\delta$, respectively.  For the sake of contradiction,
suppose that there exists a $C_4$--free subgraph $H$ such that
$\delta(H)> 2\delta/\sqrt{\Delta}$. We call a pair of edges incident
to a common vertex $v$, a \emph{cherry} of $v$. We will get a
contradiction by double counting the number of cherries of vertices of
$A$.  On the one hand, since  $\delta(H)> 2\delta/\sqrt{\Delta}$,
there are at least $\Delta\cdot\binom{2\delta/\sqrt{\Delta}}{2}$ cherries of vertices of $A$.

On the other hand, since $H$ is $C_4$--free, each pair of vertices
in $B$ has at most one common neighbor in $A$ having a cherry, thus
there are at most $\binom{\delta}{2}$ cherries of vertices of $A$, and
we have:
$$ 2\delta^2-\delta\sqrt{\Delta} = \Delta\binom{\frac{2\delta}{\sqrt{\Delta}}}{2} \leq   \binom{\delta}{2}
=\frac{ \delta^2-\delta}{2}
\;,$$
providing a contradiction.
\end{proof}

\section{Remarks and open questions}\label{sec:conc}

\begin{enumerate}
\item There are still logarithmic gaps between the lower bounds
  (Theorem~\ref{thm:edges} and Corollary~\ref{cor:deg_C4}) and the upper
  bounds for $f$ and $h$. We conjecture that the upper bounds are
  asymptotically tight.

\item In order to give a more explicit result in
  Theorem~\ref{thm:edges} it is interesting to determine the value
  $k^*=k^*(m,\Feven)$ that minimizes
  $\ex(K_{k,m/k},\Feven)$. It is clear that for every $r\geq 2$, $k^* =
  \Omega(m^{1/3})$ and $k^*=O(m^{2/3})$. Indeed, any extremal bipartite
  $\F$--free graph has at least as many edges as the size of the
  largest stable set which is $\Omega(m^{2/3})$ in both previous
  cases.

In the proof of Corollary~\ref{cor:edges_C4}, we showed that $k^*(m,C_4)=\Theta(m^{1/3})$.

Observe that
$$
\lim_{r\to\infty} k^{*}(m,\F_r)=m^{1/2}\;,
$$ When $r$ tends to infinity, $\F_r$ is composed of all
cycles of length up to $2r+1$, and thus, the extremal graph tends to a tree. In this case, the number
of edges is of the order of the number of vertices in the graph,
which is minimized when both stable sets are of the same
size approximately. Thus, we get that for every $r\geq 2$,
$$
f(m,r)= \Theta(f(m,\F_r)) = \Omega\left(\frac{m^{1/2}}{\log{m}}\right)\;.
$$ However this is meaningless since it is clear that any graph $G$
with $m$ edges has a spanning forest with at least $\Omega(m^{1/2})$
edges: such a $G$ contains a star with $\sqrt{m}$ edges, or a
matching of size $\Omega(m^{1/2})$.

\item Recently, Conlon, Fox and Sudakov showed in~\cite{cfs2014} that for every $r\geq 2$, $f(m,K_{r,r})=\Theta( m^{r/r+1})$. For $r=2$, this result improves Corollary~\ref{cor:edges_C4} by a logarithmic factor. It is worth mentioning that they were able to determine $f(m,K_{r,r})$ even though the Tur\'an numbers for $K_{r,r}$ are not known.

\item Regarding the function $h(\delta,\Delta,\F)$, we conjecture that the following holds:
\begin{conjecture}\label{conj:general}
 For every $\delta$, $\Delta$ and every family $\F$, we have
$$
h(\delta,\Delta,\F)=\Omega\left( \frac{\ex(K_\Delta,\F)\delta}{\Delta^2}\right)\;.
$$
\end{conjecture}

Proposition~\ref{prop:non_bip} shows that this conjecture is true for
$\F$ not containing any bipartite graph. The bipartite case remains wide open. Our results imply that, for any $r\geq 2$, Conjecture~\ref{conj:general} is true up to a logarithmic factor when $\F$ contains only the bipartite graphs in $\F_r^{\text{even}}$ --- if the maximum degree is not too large with respect the minimum degree. In particular, it would be interesting to set the conjecture for other important families of bipartite graphs, such as complete bipartite graphs.

\end{enumerate}

\begin{acknowledgement}
The authors would like to thank Benny Sudakov for his helpful remarks, and both referees for their work: their suggestions have improved the quality of the paper.
\end{acknowledgement}

\bibliography{girth}
\bibliographystyle{amsplain}

\end{document}